\documentclass[11pt]{amsart}
\usepackage[latin1]{inputenc}
\usepackage{youngtab}
\usepackage{fullpage}
\usepackage{xspace,amssymb,amsfonts,euscript,}
\usepackage{amsthm,amsmath}
\usepackage{amscd}
\usepackage{palatino}
\usepackage{euscript}
\usepackage{mathtools}
\usepackage{amsmath}
\usepackage{amsfonts}
\usepackage{amssymb}
\usepackage{graphicx}

\makeatletter
 \def\@textbottom{\vskip \z@ \@plus 8pt}
 \let\@texttop\relax
\makeatother

\providecommand{\U}[1]{\protect\rule{.1in}{.1in}}

\numberwithin{equation}{section}
\input xy
\xyoption {all}
\RequirePackage{color}
\definecolor{myred}{rgb}{0.75,0,0}
\definecolor{mygreen}{rgb}{0,0.5,0}
\definecolor{myblue}{rgb}{0,0,0.65}

\RequirePackage{ifpdf}
\ifpdf
\IfFileExists{pdfsync.sty}{\RequirePackage{pdfsync}}{}
\RequirePackage[pdftex,
colorlinks = true,
urlcolor = myblue,    citecolor = mygreen,    linkcolor = myred,
bookmarksopen=true]{hyperref}
\else
\RequirePackage[hypertex]{hyperref}
\fi
\RequirePackage{ae, aecompl}

\newcommand\tsqbinom[2]{\genfrac{[}{]}{0pt}{1}{#1}{#2}}

\def\ind{\boldsymbol 1}

\def\lamb{{\boldsymbol{\lambda}}}
\def\mub{{\boldsymbol{\mu}}}

\def\Spr{{\mathrm{Spr}}}

\def\BG{{\mathfrak B}}

\def\CM{{\mathbb{C}}}

\def\EM{{\mathbb{E}}}

\def\FM{{\mathbb{F}}}

\def\gg{{\mathfrak g}}

\def\KM{{\mathbb{K}}}

\def\NG{{\mathfrak N}}

\def\NM{{\mathbb{N}}}

\def\OM{{\mathbb{O}}}
\def\PG{{\mathfrak P}}

\def\QM{{\mathbb{Q}}}

\def\SG{{\mathfrak S}}

\def\BC{{\mathcal{B}}}

\def\DB{{\mathbf D}}

\def\LC{{\mathcal{L}}}

\def\NC{{\mathcal{N}}}

\def\OC{{\mathcal{O}}}

\def\PC{{\mathcal{P}}}

\def\SB{{\mathbf S}}

\def\LS{{\EuScript L}}

\newcommand{\nc}{\newcommand}
\newcommand{\renc}{\renewcommand}
\newcommand{\rdots}{\mathinner{ \mkern1mu\raise1pt\hbox{.}
\mkern2mu\raise4pt\hbox{.}
\mkern2mu\raise7pt\vbox{\kern7pt\hbox{.}}\mkern1mu}}

\def\co{{\mathrm{co}}}

\def\triv{{\mathrm{triv}}}

\DeclareMathOperator{\Irr}{Irr}

\DeclareMathOperator{\Part}{{\mathbf{Part}}}
\DeclareMathOperator{\Bipart}{{\mathbf{Bip}}}
\DeclareMathOperator{\Unbipart}{\overline{\Bipart}}
\DeclareMathOperator{\Dpart}{\widetilde{\Bipart}}
\DeclareMathOperator{\Symb}{{\mathbf{Symb}}}
\DeclareMathOperator{\DSymb}{{\mathbf{DSymb}}}

\DeclareMathOperator{\Sgn}{\varepsilon}

\def\Sp{{\mathbf{Sp}}}
\def\SO{{\mathbf{SO}}}
\def\GL{{\mathbf{GL}}}
\def\pr{{\mathrm{pr}}}

\def\ov{\overline}

\def\to{\rightarrow}

\nc{\triright}{\stackrel{[1]}{\to}}
\nc{\Br}{\mathcal{B}}
\nc{\HotRR}{{}_R\mathcal{K}_R}
\nc{\HotR}{\mathcal{K}_R}
\nc{\excise}[1]{}
\nc{\defect}{\text{df}}
\nc{\h}[1]{\underline{H}_{#1}}
\nc{\Ga}{\mathbb{G}_a}
\nc{\Gm}{\mathbb{G}_m}
\nc{\Perv}{{\mathbf{P}}}
\nc{\IH}{{\mathrm{IH}}}
\nc{\ic}{\mathbf{IC}}
\nc{\gl}{{\mathfrak{gl}}}
\renc{\sl}{{\mathfrak{sl}}}
\renc{\sp}{{\mathfrak{sp}}}

\DeclareMathOperator{\im}{{\mathrm{Im}}}

\DeclareMathOperator{\Res}{Res}

\newtheorem{theorem}{Theorem}[section]
\newtheorem{lemma}[theorem]{Lemma}
\newtheorem{proposition}[theorem]{Proposition}
\newtheorem{corollary}[theorem]{Corollary}
\newtheorem{example}[theorem]{Example}

\theoremstyle{definition}
\newtheorem{definition}[theorem]{Definition}
\theoremstyle{remark}
\newtheorem{remark}[theorem]{Remark}

\def\bfl{{\boldsymbol \lambda}}
\def\bfm{{\boldsymbol \mu}}

\DeclareMathOperator{\dist}{dist}

\begin{document}

\begin{abstract}
We define the notion of basic set data for finite groups (building on the
notion of basic set, but including an order on the irreducible characters as part of the structure), and we prove that the Springer correspondence provides basic set data for Weyl groups. Then we use this to determine explicitly the modular Springer correspondence for classical types (defined over a base field of odd characteristic $p$, and with coefficients in a field of odd characteristic $\ell\neq p$): the modular case is obtained as a restriction of the ordinary case to a basic set. In order to do so, we compare the order on bipartitions introduced by Dipper and James with the order induced by the Springer correspondence. We also provide a quicker proof, by sorting characters according to the dimension of the corresponding Springer fiber, an invariant which is directly computable from symbols.
\end{abstract}

\title[Modular Springer correspondence for classical types]{Springer basic sets and modular Springer correspondence for 
classical types}

\author{Daniel Juteau}
\address{CNRS UMR 7586, Insitut de Math\'ematques de Jussieu--Paris Rive Gauche, Universit\'e de Paris, B\^atiment Sophie Germain, 8 Place Aur\'elie Nemours, 75013 Paris, France}
\email{daniel.juteau@cnrs.fr}
\urladdr{https://webusers.imj-prg.fr/~daniel.juteau/}

\author{C\'edric Lecouvey}
\address{CNRS UMR 7013, Institut Denis Poisson, UFR Sciences et Techniques, Universit\'e Fran\c cois Rabelais, Parc de
Grandmont, 37200 Tours, France }
\email{cedric.lecouvey@lmpt.univ-tours.fr}
\urladdr{http://www.lmpt.univ-tours.fr/~lecouvey/}

\author{Karine Sorlin}
\address{
CNRS UMR 7352,
Laboratoire Ami\'enois de Math\'ematique Fondamentale et Appliqu\'ee,
UFR des Sciences,
33, rue Saint-Leu,
80039 Amiens Cedex 1,
France }
\email{karine.sorlin@u-picardie.fr}
\urladdr{http://www.lamfa.u-picardie.fr/sorlin/}

\thanks{This research was partially supported by the ANR projects REPRED (ANR-09-JCJC-0102-01) and GEREPMOD (ANR-16-CE40- 0010-01). The first author is gratedul for the support of the Charles Simonyi Endowment at the Institute for Advanced Study, where the last stages of the writing were accomplished.}

\maketitle

\begin{center}
\emph{To the memory of Tonny Springer.}
\end{center}

\section{Introduction}

\subsection{Context}

Let $G$ be a connected reductive algebraic group defined over $\ov\FM_{p}$, and let $W$ be its Weyl group. Assume $p$ is good. If $\OC$ is a nilpotent orbit, we denote by $\BC_\OC$ the fibre of the Springer resolution of the nilpotent cone \cite{Springer:unip} at an element $x\in\OC$. In the landmark paper \cite{Springer:Trig}, Springer defined an action of $W$ on the $\ell$-adic cohomology of $\BC_\OC$, and a correspondence which associates to each irreducible representation of $W$ a pair consisting of a nilpotent orbit and a $G$-equivariant $\ell$-adic local system on that orbit.

The Springer correspondence is an injection rather than a bijection: some pairs are missing; it was later generalized by Lusztig, to include all pairs \cite{Lusztig:ICC}. Actually he works in the group case (with unipotent
classes rather than nilpotent orbits), and for arbitrary $p$. Note that when $p$ is good, the classification of unipotent classes or nilpotent orbits is independent of $p$, and we have a $G$-equivariant Springer homeomorphism from the unipotent variety to the nilpotent cone \cite{Springer:unip}; for those $p$, the Springer correspondence is independent of $p$.

The non-generalized correspondence (for the group) was explicitly determined
in all types \cite{Springer:Trig, Hotta-Springer, Shoji:classical-Springer,
Shoji:F4-Springer, Alvis-Lusztig}, and also the generalized correspondence
\cite{Lusztig:ICC, Lusztig-Spaltenstein:classical-Springer,
Spaltenstein:exceptional-Springer,Lusztig:CS,Lusztig:GSC} (there is one last ambiguity in type $E_8$ for $p = 3$). For $p$ bad, besides the group
case one can consider the Lie algebra or its dual, and all three situations are different \cite{Xue}.

The Springer correspondence and its generalization are closely related to Green functions. The orthogonality relations for 
generalized Green functions led to an algorithm, due to Lusztig and Shoji, to compute the stalks of the intersection 
cohomology complexes of all nilpotent orbits, with coefficients in any equivariant local system \cite[\S 24]{Lusztig:CS}.

In \cite{Juteau:thesis}, a modular Springer correspondence was defined,
following the Fourier-Deligne transform approach of \cite{Brylinski}. That is, one
considers irreducible representations of the Weyl groups over a field
$\FM$ of characteristic $\ell\neq p$, and local systems of vector spaces
over $\FM$. Besides, one can see the decomposition matrix of a Weyl
group as a submatrix of a decomposition matrix for equivariant perverse
sheaves on the nilpotent cone defined in \cite{Juteau:decperv}. The modular Springer correspondence was explicitly determined in \cite{Juteau:thesis} in
type $A_{n}$, and also in other types of rank at most three. These results are recalled in \cite{Juteau:modspringer}, together with the explicit determination of the modular Springer correspondence for exceptional Weyl groups. See also \cite{AHJR:Weyl} for a construction of the modular Springer correspondence for complex reductive Le algebras, and \cite{AHJR1,AHJR2,AHJR3} for the generalized case.

In the present paper, we determine the modular Springer correspondence for
classical types when $\ell\neq 2$. Since the case $\ell = 2$ (generalized) was solved in \cite[Corollary 9.8]{AHJR2} (for $\gg$ over $\CM$ rather than $\ov\FM_p$ but it does not really matter), and the exceptional cases were done in \cite{Juteau:modspringer}, the modular Springer correspondence is therefore completely determined in all types (for $p$ good, but for any $\ell \neq p$). As for the combinatorics of the generalized modular Springer correspondence for classical Lie algebras in the case $\ell\neq 2$, the implicit conjectures in \cite[\S 7--8]{AHJR2} (i.e.~in each case, the combinatorial bijection which is used to count cuspidals should be the generalized Springer correspondence) remain to be proved. 

\subsection{A diagram}
\label{subsec:diag}
Our strategy is to use the known results in characteristic zero \cite{Shoji:classical-Springer,Lusztig:ICC}, and unitriangularity properties of the decomposition matrices (both for the Weyl groups and for the perverse sheaves) so that there is only one possibility for the correspondence. The whole paper is summed up in the following diagram, whose meaning will be gradually explained.

\[
\xymatrix@R=1.4cm@C=1.8cm{
\Irr \KM W 
	\ar@{=}[r] &
\Bipart_W 
	\ar[r]^{\mbox{$\bfl \mapsto \bfl^{*}$}}
	\ar@{^{(}->}[dr]_-{\mbox{$\Psi^{G}_{\KM}$}} &
\Bipart_W 
	\ar@{^{(}->}[d]_-{\mbox{$\Phi^{G}_{\KM}$}}
	\ar@{^{(}->}[r]^-{\mbox{$\Phi_\co^{G}$}}&
\Symb_{G} 
	\ar@/^.5em/[ld]_{\mbox{$\Pi^G$}}
	\ar@{>>}@<1ex>[d] ^{\dist}
	\ar@/^2em/[ddr]^{{\mbox{$\delta_\co$}}}&
\\
\Irr \FM W  
	\ar@{^{(}->}[u]^{ \mbox{$\beta_{\ell} = \gamma_{\ell} $}}  
	\ar@{=}[r] &
\Bipart^{(\ell)}_W 
	\ar@{^{(}->}[u]
	\ar@{^{(}->}[r]_{\mbox{$\Psi^{G}_{\FM}$}} &
\PG_{G}
	\ar@{>>}@<.7ex>[d]^{\pr} 
	\ar@{-->}@/_.4em/^{\mbox{$\delta$}}[drr] &
\DSymb_{G} 
	\ar[ld]_(.6){\mbox{$\Pi^G_1$}}
	\ar@{^{(}->}@<.7ex>[u] 
	\ar[dr]^{{\mbox{$\delta_\co$}}}&
\\
&&
\NG_{G}
	\ar@{^{(}->}@<.7ex>[u]^{\triv}
	\ar@/_.8em/_{\mbox{$\delta: \OC \longmapsto \dim \BC_\OC$}}[rr] &  & 
\NM
}
\]

Let us first briefly describe the objects in the diagram. First, $W$ is the Weyl group of a classical group $G$ over $\ov\FM_p$ with $p$ odd. In the first column, we see its ordinary and modular characters ($\KM$ has characteristic $0$, while $\FM$ has characteristic $\ell\neq 2,p$); the second column contains their parametrizations by bipartitions (there are variants according to the types, but we use the generic notation $\Bipart_W$). In the next column the geometry appears: $\PG_G$ represents the set of pairs consisting of a nilpotent orbit together with an irreducible $G$-equivariant local system; we identify the ordinary and modular cases, because all component groups of centralizers are $2$-groups and $\ell$ is odd. The first projection to the set of nilpotent orbits $\NG_G$ has a section assigning the trivial local system to each nilpotent orbit. In the next column we see a set of symbols associated to $G$, which are very useful to describe the Springer correspondence, as they can parametrize both characters and local systems \cite{Lusztig:ICC,Geck-Malle:support}. There is an equivalence relation on symbols such that similar symbols correspond to different local systems on the same nilpotent orbit; the subset $\DSymb_G$ of distinguished symbols corresponds to the trivial local systems.

Let us now comment on the various Springer correspondences that appear. The arrow $\Psi^{G}_{\KM}$ stands for the ordinary correspondence obtained by Fourier transform, as in \cite{Springer:Trig,Brylinski}, and $\Psi^{G}_{\FM}$ is its modular analogue studied in \cite{Juteau:thesis,Juteau:modspringer,AHJR:Weyl}. However, the ordinary Springer correspondence is most conveniently expressed (via $\Phi^G_\co$) in its other version $\Phi^{G}_{\KM}$, which can be obtained by restricting intersection cohomology complexes to the nilpotent cone \cite{Lusztig:Green,Borho-MacPherson:CRAS}. The two versions differ by the tensoring with the sign character $\Sgn_{W}$, which is reflected combinatorially by transposing the bipartitions. Only some pairs and symbols will be relevant for us:
$\PG_G^\Spr := \im\Phi^G_\KM\supset\im\triv$ and
$\Symb_G^\Spr := \im \Phi^G_\co\supset\DSymb_G$.

\subsection{Outline}
The unitriangularity properties of the decomposition matrices are related to the notion of \emph{basic sets}, which are subsets of the set of ordinary irreducible characters. Usually an order on the characters is fixed. We define the more precise notion of \emph{basic set datum}, where the order is the first half of the structure, the other half being an injection from the
modular characters to the ordinary characters. This is because we will consider several orders simultaneously. The definition and first properties of basic set data are the subject of \S \ref{sec:basic set data}. The main point is that if we have two basic set data whose order relations are comparable (i.e. one is finer than the other), then the injection has to be the same.

In \S \ref{sec:springer basic sets}, we will see that the Springer correspondence gives rise to basic set data for Weyl 
groups. The situation is much simpler when the characteristic $\ell$ does not divide the orders of the components groups of 
the centralizers of the nilpotent elements: one may identify ordinary and modular local systems on nilpotent orbits with a common indexing set $\PG^{G}$, and the Springer order is directly induced by the closure inclusion of nilpotent orbits 
(local systems on the same nilpotent orbit are not comparable). This defines a canonical injection $\beta_\ell: \Irr \FM W \hookrightarrow \Irr \KM W$. As an illustration, we re-derive the modular Springer correspondence for general linear groups using this language, noting that the order on nilpotent orbits and the order relevant to the modular representation theory of the symmetric group are the same: the dominance order on partitions.

From \S \ref{sec:DJ}, we consider only groups of classical types (over $\ov\FM_{p}$ with $p$ odd). We recall results of Dipper and James, which in particular provide a basic set datum, for a very coarse order. In particular, we get an injection $\gamma_{\ell}:  \Irr \FM W \hookrightarrow \Irr \KM W$. At the combinatorial level, this is the inclusion of $\ell$-regular bipartitions into all regular bipartitions (for type $D_{n}$, one has to take unordered bipartitions, possibly with decorations; to have a uniform notation, we call $\Bipart_{W}$ the set of bipartitions which is relevant to $W$).

We will see that the modular Springer correspondence $\Psi^{G}_{\FM}$ can be described by composing $\beta_{\ell}$ with the ordinary Springer correspondence $\Psi^{G}_{\KM}$. Our main theorem says that $\beta_{\ell} = \gamma_{\ell}$, which gives a concrete description of the modular Springer correspondence. In order to prove this, we want to compare the Dipper--James order with the one induced by the Springer correspondence. However, it is easier to compare both of them to a third order, given by the dimensions of the Springer fibres, which have a combinatorial description by \cite{Geck-Malle:support}. In \S \ref{sec:symbols}, we recall the combinatorics of symbols, and the function $\delta_\co$ giving the dimensions of the Springer fibres in terms of symbols. Then in \S \ref{sec:mod}, we prove that the Dipper--James order is compatible with the $\delta_{\co}$ function. Since the Springer order also satisfies this property, we can conclude.

So the trick of using $\delta_{\co}$ allowed us to prove our main theorem with a minimal amount of combinatorics, bypassing the link between symbols and nilpotent orbits. However, it was still an interesting question to determine whether there is a direct relationship between the Dipper--James and Springer orders. We are able to do so in \S \ref{sec:comb}.

\subsection*{Acknowledgements}

During our investigations, it was very helpful to use the development version
\cite{Michel} of the CHEVIE \cite{chevie} package of GAP3 \cite{GAP3}: it contains all
the data of the ordinary (generalized) Springer correspondence. We
thank Jean Michel for his help with programming.

\section{Basic set data}

\label{sec:basic set data}

Let $W$ be a finite group, $\ell$ a prime number, and $(\KM, \OM, \FM)$ a sufficiently large $\ell$-modular system, e.g. we 
can take for $\KM$ a finite extension of $\QM_{\ell}$ containing $m$-th roots of unity, where $m$ is the exponent of $W$, 
with ring of integers $\OM$ and residue field $\FM$. We will denote by $\EM W$ the group algebra of $W$ over $\EM$, and 
for $\EM = \KM$ or $\FM$, we write $\Irr \EM W$ for a set of representatives of isomorphism classes of simple $\EM 
W$-modules. 
We have an $\ell$-modular decomposition matrix
\[
D^{W} := (d^{W}_{E,F})_{E\in\Irr \KM W,\ F \in\Irr \FM W}
\]
where $d^{W}_{E,F}$ is the composition multiplicity of the simple module $F$
in $\FM \otimes_\OM E_\OM$, where $E_\OM$
is some integral form of $E$. This is independent of the choice of
$E_\OM$ \cite[\S 15.2]{Serre}.

\begin{definition}

A basic set datum (for $W$) is a pair ${\mathfrak{B}} = (\leq,\beta)$ consisting of a partial order $\leq$ on $\Irr \KM W$, and 
an injection $\beta: \Irr \FM W \hookrightarrow\Irr \KM W$ such that the following properties hold:
\begin{align}
\label{eq:uni}
d^{W}_{\beta(F),F} = 1 &\text{ for all } F\in\Irr\FM W;\\
\label{eq:tri}
d^{W}_{E,F} \neq0 \Rightarrow E \leq\beta(F) &\text{ for all } E\in\Irr\KM W,\ F\in\Irr\FM W.
\end{align}

\end{definition}

We will need to compare different basic set data. We adopt the convention that
an order relation on a set $X$ is a subset of the product $X \times X$. Thus
we can consider the intersection of two orders, say that one order is included
in another one, and so on.

\begin{proposition}
\label{prop:beta fixed} Let us consider a fixed injection $\beta:
\Irr \FM W \hookrightarrow\Irr \KM W$.

\begin{enumerate}
\item If $(\leq_{1},\beta)$ is a basic set datum and $\leq_{1} \subseteq
\leq_{2}$, then $(\leq_{2},\beta)$ is also a basic set datum.

\item If $(\leq_{1},\beta)$ and $(\leq_{2},\beta)$ are two basic set data,
then $(\leq_{1} \cap\leq_{2},\beta)$ is also a basic set datum.
\end{enumerate}
\end{proposition}

\begin{proof}
This follows directly from the definitions.
\end{proof}

\begin{proposition}
\label{prop:comp}
Let ${\mathfrak{B}}_{1} = (\leq_{1},\beta_{1})$ and
${\mathfrak{B}}_{2} = (\leq_{2},\beta_{2})$ be two basic set data for $W$.
Then, for any $F\in\Irr\FM W$, we have
\[
\beta_{1}(F) \leq_{2} \beta_{2}(F) \leq_{1} \beta_{1}(F).
\]
If moreover $\leq_{1} \subseteq\leq_{2}$, then $\beta_{1} = \beta_{2}$.
\end{proposition}

\begin{proof}
Let $F$ be an element of $\Irr\FM W$. By the property \eqref{eq:uni}
for ${\mathfrak{B}}_{1}$, we have $d^{W}_{\beta_{1}(F),F} = 1 \neq0$, hence by
the property \eqref{eq:tri} for ${\mathfrak{B}}_{2}$, we have $\beta_{1}(F)
\leq_{2} \beta_{2}(F)$. Similarly, by the property \eqref{eq:uni} for
${\mathfrak{B}}_{2}$, we have $d^{W}_{\beta_{2}(F),F} = 1 \neq0$, hence by the
property \eqref{eq:tri} for ${\mathfrak{B}}_{1}$, we have $\beta_{2}(F)
\leq_{1} \beta_{1}(F)$. Thus $\beta_{1}(F) \leq_{2} \beta_{2}(F) \leq_{1}
\beta_{1}(F)$.

If moreover $\leq_{1} \subseteq\leq_{2}$, then $\beta_{1}(F) \leq_{2}
\beta_{2}(F) \leq_{2} \beta_{1}(F)$, hence $\beta_{1}(F) = \beta_{2}(F)$, for
any $F$.
\end{proof}

\begin{corollary}
\label{cor:unique beta} Given a partial order $\leq$ on $\Irr \KM W$,
there is at most one injection $\beta: \Irr\FM W \hookrightarrow
\Irr \KM W$ such that $(\leq,\beta)$ is a basic set datum.
\end{corollary}

\section{Springer basic set data}

\label{sec:springer basic sets}

\subsection{Modular Springer correspondence}

\label{subsec:modular springer correspondence}

Let us consider a reductive group $G$ over $\overline\FM_{p}$. We assume that $p$ is very good for $G$ \cite[p.29]{Carter} (thus for classical groups of type $BCD$, we only
have to assume $p \neq2$). The group $G$ acts by the adjoint action on its Lie algebra ${\mathfrak{g}}$, and there are finitely 
many orbits in the nilpotent cone $\NC \subseteq{\mathfrak{g}}$. For $x$ in $\NC$, we denote by $\OC_{x}$ the $G$-orbit of 
$x$. The set ${\mathfrak{N}}_{G}$ of nilpotent orbits is partially ordered by the following relation: $\OC_{1} \leq\OC_{2}$ if 
and only if $\OC_{1} \subseteq\overline\OC_{2}$. For each orbit
$\OC \subseteq\NC$, we choose a representative $x_\OC$, and we denote by $A_{G}(\OC)$ the group of
components of the centralizer $C_{G}(x_\OC)$, that is, $A_{G}(\OC) := C_{G}(x_\OC)/C_{G}(x_\OC)^0$.

As in the last section, we choose an $\ell$-modular system $(\KM, \OM, \FM)$, and we assume $\ell\neq p$. For $\EM=\KM$ 
or $\FM$, let $\PG_\EM^{G}$ (or $\PG_\EM$ if $G$ is clear from the context) denote the set of all pairs $(\OC,\LC)$ where 
$\OC$ 
is a nilpotent orbit in ${\mathfrak{g}}$ and $\LC$ is an irreducible $G$-equivariant $\EM$-local system on $\OC$. 
Moreover, we will denote by $\LC_{\rho}$ the local system corresponding to some $\rho\in\Irr\EM A_G(\OC)$. 
The simple $G$-equivariant perverse sheaves on $\NC$ with $\EM$-coefficients are intersection cohomology complexes 
$\ic(\overline\OC,\LC)$, 
for $(\OC,\LC)\in\PG_\EM$. Similarly to the case of a finite group, one can define a decomposition 
matrix \cite{Juteau:decperv}
\[
D^\NC:=( d_{(\OC,\LC_{\rho}),(\OC',\LC_{\sigma})}^\NC )_{(\OC,\LC_{\rho})\in\PG_\KM,\ (\OC',\LC_{\sigma})\in\PG_\FM}
\]
for $G$-equivariant perverse sheaves on $\NC$, where $d_{(\OC,\LC),(\OC',\LC')}^\NC$ is the composition multiplicity of 
$\ic(\ov\OC',\LC')$ 
in $\FM\otimes_\OM^{L}\ic(\ov\OC,\LC_\OM)$ for any
choice of an integral form $\LC_\OM$ of the local system $\LC$.

Let $W$ be the Weyl group of $G$. If $\EM = \KM$ or $\FM$,
one can naturally define an injective map
\begin{align}
\Psi_\EM = \Psi^{G}_{\EM}: \Irr \EM W \hookrightarrow\PG_\EM\label{def:SpringerCorrespondence}%
\end{align}
using a Fourier-Deligne transform. In the ordinary case, it is the Springer
correspondence of \cite{Brylinski} (and coincides with the one of
\cite{Springer:Trig}), and in the modular case it was first introduced in
\cite{Juteau:thesis}. See also \cite{AHJR:Weyl} for complex reductive Lie algebras.

\begin{remark}
The injectivity of $\Psi_\EM$ comes from the fact that there is an inverse Fourier transform.
\end{remark}

\begin{remark}
We have $\Psi_\EM(\EM) = (\{0\},\EM)$ because the Fourier-Deligne
transform of the constant sheaf is a sky-scraper sheaf supported at the
origin, just as the Fourier transform of a constant function is a Dirac
distribution supported at the origin.
\end{remark}

\begin{remark}
\label{rem:phi} What is usually called Springer correspondence (but only in the case $\EM = \KM$) is the one defined in
\cite{Lusztig:Green,Borho-MacPherson:CRAS}, and for which one can find a
description in all types in \cite{Carter}. We will denote this map by
\[
\Phi_\KM = \Phi_\KM^G: \Irr \KM W \hookrightarrow{\PG}_\KM.
\]
Note that this construction makes sense both for the group $G$ and for its Lie algebra ${\mathfrak{g}}$. Since the characteristic $p$ is very good for $G$, we can also use a Springer isomorphism \cite{Springer:unip} from the unipotent variety to the nilpotent cone to transfer the contruction for the group to the Lie algebra, but by \cite[5.3]{Shoji:asterisque} this gives the same result as working directly with the Lie algebra.
\end{remark}

The maps $\Phi_\KM$ and $\Psi_\KM$ differ by the sign
character $\Sgn_{W}$ of the Weyl group:

\begin{proposition}
\cite[17.7]{Shoji:asterisque} \label{prop:sign}
For any $E\in\Irr \KM W$, we have
\[
\Psi_\KM(E) = \Phi_\KM(\Sgn_{W}\otimes E).
\]
\end{proposition}

See also \cite{AHJR:Weyl} for the same result in the modular case.
The following theorem from \cite{Juteau:thesis} shows that the ordinary and
modular Springer correspondences respect decomposition numbers.

\begin{theorem}
\label{th:dec} For $E\in\Irr \KM W$ and $F\in\Irr\FM W$, we
have
\[
d^{W}_{E,F} = d^\NC_{\Psi_\KM(E),\Psi_\FM(F)}.
\]

\end{theorem}

Hence $D^{W}$ can be seen as a submatrix of $D^\NC$. Now
$D^\NC$ has the following easy property, as follows from considering
supports \cite[Corollary 2.6 and Proposition 2.7]{Juteau:modspringer}.

\begin{proposition}
\label{prop:support} For $(\OC,\LC_{\rho}) \in\PG_\KM$ and $(\OC',\LC_{\sigma}) \in\PG_\FM$, we have
\begin{align}
d^\NC_{(\OC,\LC_{\rho}),(\OC',\LC_{\sigma})}  & = 0 \text{ unless } \OC' \leq \OC,\\
d^\NC_{(\OC,\LC_{\rho}),(\OC',\LC_{\sigma})}  & = d^{A_{G}(\OC)}_{\rho,\sigma} \text{ if } \OC = \OC'.
\end{align}

\end{proposition}

In the preceding formula, we have denoted by $d^{A_{G}(\OC)}_{\rho,\sigma}$ the
decomposition numbers for the finite group $A_{G}(\OC)$. In order to define
basic sets using the Springer correspondence, the following result plays a
crucial role \cite[Proposition 5.4]{Juteau:modspringer}.

\begin{proposition}
\label{prop:principal-series} If $(\OC',\LS_{\sigma}) \in\im\Psi_\FM$ and $d^\NC_{(\OC,\LC_{\rho}),(\OC',\LC_{\sigma})} \neq 
0$, then $(\OC,\LC_{\rho}) \in \im\Psi_\KM$.
\end{proposition}

\subsection{Definition of Springer basic sets}

\label{subsec:definition}

Suppose that we have a family ${\mathfrak{B}}_{\bullet}= ({\mathfrak{B}}_\OC)$ of basic set data for the groups 
$A_{G}(\OC)$. 
For each nilpotent orbit $\OC$, we write ${\mathfrak{B}}_\OC = (\leq_\OC,\beta_\OC)$. To such a family 
${\mathfrak{B}}_{\bullet}$, 
we will associate a basic set datum 
$(\leq^{\mathcal{N}}_{{\mathfrak{B}}_{\bullet}},\beta^\NC_{{\mathfrak{B}}_{\bullet}})$ 
for $W$, the \emph{Springer basic set datum}.
%We will call the image of the injection
%$\beta^\NC_{\BG_\bullet}$ a \emph{Springer basic set}.

\begin{remark}
If $G$ is simple and of adjoint type (which we may assume for the purpose of
describing the correspondence, since we are working in a modular, but not
generalized setting, in the sense of Lusztig), then the $A_{G}(\OC)$
are as follows: trivial in type $A_{n}$, elementary abelian $2$-groups in
other classical types, and symmetric groups of rank at most $5$ in exceptional
types. We know that these groups admit basic sets. In the sequel it will be
combinatorially more convenient to work with special orthogonal and symplectic
groups (rather than the adjoint groups), but for those the $A_{G}(\OC)$ are also elementary abelian $2$-groups.
\end{remark}

\begin{remark}
In the case where $\ell$ does not divide the orders of the groups $A_{G}(x)$,
we will see that much of the discussion simplifies (see \S
\ref{subsec:good}). In particular, there will be no need to choose basic set
data for each orbit: in that case, we may assume that we take the trivial
order everywhere, and the $\beta_\OC$ are canonical bijections (in
that case, the Brauer characters are just the ordinary characters).
\end{remark}

\begin{definition}
We define the partial order $\leq^\NC_{{\mathfrak{B}}_{\bullet}}$ on
$\Irr\KM W$ by the following rule. For $i \in\{1,2\}$, let $E_{i}
\in\Irr\KM W$, and let us write $\Psi_\KM(E_{i}) =
(\OC_{i},\LC_{\rho_i})$. Then
\[
E_{1} \leq^\NC_{{\mathfrak{B}}_{\bullet}} E_{2} \Longleftrightarrow%
\begin{cases}
\text{either} & \OC_{2} < \OC_{1},\\
\text{or} & (\OC_{1} = \OC_{2} =: \OC \text{ and
} \rho_{1} \leq_\OC \rho_{2} ).
\end{cases}
\]

\end{definition}

It is clear that this defines indeed an order relation on $\Irr \KM
W$ (for the antisymmetry, one needs to use the injectivity of $\Psi
_\KM$).

\begin{proposition}
\label{prop:F->E} Let $F \in\Irr \FM W$, and let us write
$\Psi_\FM(F) = (\OC,\LC_\sigma)$. Then there exists a unique $E
\in\Irr\KM W$ such that $\Psi_\KM(E) = (\OC,
\LC_{\beta_\OC(\sigma)})$.
\end{proposition}

\proof{
We have
$(\OC,\LC_\sigma) = \Psi_\FM(F) \in \im \Psi_\FM$,
and
\[
d^\NC_{(\OC,\LC_{\beta_\OC(\sigma)}),(\OC,\LC_\sigma)} =
d^{A_G(\OC)}_{\beta_\OC(\sigma),\sigma} = 1\neq 0,
\]
hence $(\OC, \LC_{\beta_\OC(\sigma)}) \in \im \Psi_\KM$ by Proposition
\ref{prop:principal-series}. This proves the existence of $E$.
The uniqueness of $E$ follows from the injectivity of $\Psi_\KM$.
\qed
}

\begin{definition}
We define the map $\beta^\NC_{{\mathfrak{B}}_{\bullet}} :
\Irr\FM W \rightarrow\Irr \KM W$ by associating to each
$F\in\Irr\FM W$ the module $E\in\Irr\KM W$ provided by
Proposition \ref{prop:F->E}. \label{def:Springerbeta}
\end{definition}

In other words, $\beta^\NC_{{\mathfrak{B}}_{\bullet}}$ completes the
following commutative diagram :
\[
\xymatrix{ \Irr \FM W \ar[r]^(.6){\Psi_\FM} \ar[d]_{\beta^\NC_{\BG_\bullet}} & \PG_\FM \ar[d]^{\beta_\bullet} \ar[r]^(.3){\sim} & 
\sqcup_\OC \Irr \FM A_G(\OC) \ar[d]^{\sqcup_\OC \beta_\OC} \\ \Irr \KM W \ar[r]_(.6){\Psi_\KM} & \PG_\KM \ar[r]_(.3){\sim} & 
\sqcup_\OC \Irr \KM A_G(\OC)\\ }
\]
where $\beta_{\bullet}$ is defined by the commutativity of the right square.
The $\beta_\OC$ are injective, thus $\beta_{\bullet}$ is injective,
and the composition $\beta_{\bullet}\circ\Psi_\FM = \Psi_\KM
\circ\beta^\NC_{{\mathfrak{B}}_{\bullet}}$ is injective as well. It
follows that $\beta^\NC_{{\mathfrak{B}}_{\bullet}}$ is injective.

\begin{theorem}
\label{th:SpringerBasicSet} Assume we have chosen a family ${\mathfrak{B}%
}_{\bullet}= ({\mathfrak{B}}_\OC)$ of basic set data for the groups
$A_{G}(\OC)$, where $\OC$ runs over all nilpotent orbits.
Then the pair $(\leq^\NC_{{\mathfrak{B}}_{\bullet}}, \beta
^\NC_{{\mathfrak{B}}_{\bullet}})$ defined above is a basic set datum
for $W$.
\end{theorem}

\proof{
Let $F \in \Irr \FM W$. Let us write $(\OC',\LC_\sigma) := \Psi_\FM(F)$,
and let $E_F := \beta^\NC_{\BG_\bullet}(F)$. By definition, we have
$\Psi_\KM(E_F) = (\OC', \LC_{\beta_{\OC'}(\sigma)})$. Thus
\[
d^W_{\beta^\NC_{\BG_\bullet}(F),F} = d^W_{E_F,F}
= d^\NC_{(\OC', \LC_{\beta_\OC(\sigma)}),(\OC',\LC_\sigma)}
= d^{A_G(\OC')}_{\beta_{\OC'}(\sigma),\sigma}
= 1,
\]
using Theorem \ref{th:dec} and Proposition \ref{prop:support}.

Now, in addition, let $E \in \KM W$, and assume $d^W_{E,F} \neq 0$.
Let us write $(\OC, \LC_\rho) := \Psi_\KM(E)$.
By Theorem \ref{th:dec} again, we have
\[
d^\NC_{(\OC, \LC_\rho),(\OC', \LC_\sigma)} = d^W_{E,F} \neq 0
\]
hence $\OC \geq \OC'$ by Proposition \ref{prop:support}.
Now, if $\OC = \OC'$, then we have
\[
0 \neq d^\NC_{(\OC,\LC_\rho),(\OC,\LC_\sigma)} = d^{A_G(\OC)}_{\rho,\sigma}
\]
again by Proposition \ref{prop:support}, hence
$\rho \leq_\OC \beta_\OC(\sigma)$ since
$\BG_\OC = (\leq_\OC,\beta_\OC)$ is a basic set for $A_G(\OC)$.
This proves that $E \leq^\NC_{\BG_\bullet} E_F = \beta^\NC_{\BG_\bullet}(F)$.
Thus $(\leq^\NC_{\BG_\bullet}, \beta^\NC_{\BG_\bullet})$
is a basic set datum for $W$.
\qed
}

\bigskip

The pair $(\leq^\NC_{{\mathfrak{B}}_{\bullet}}, \beta^{\mathcal{N}%
}_{{\mathfrak{B}}_{\bullet}})$ is the Springer basic set datum for $W$
associated to the choice of the family ${\mathfrak{B}}_{\bullet}$, and we call
the image of $\beta^\NC_{{\mathfrak{B}}_{\bullet}}$ a Springer basic set.

\subsection{Good case}

\label{subsec:good}

In this paragraph, we assume that $\ell$ does not divide the orders of the
groups $A_{G}(\OC)$, where $\OC$ runs over the nilpotent
orbits in ${\mathfrak{g}}$. Then the algebras $\FM A_{G}%
(\OC)$ are semisimple and the decomposition matrices $D^{A_{G}%
(\OC)}$ are equal to the identity matrix. Thus we can take the
trivial order (that is, the identity relation) $=_\OC$ on
$\Irr \KM A_{G}(\OC)$, and $\beta_\OC$ is a
canonical bijection. By abuse of notation, we will identify the sets
$\PG_\KM$ and $\PG_\FM$ with a common
index set $\PG$. Let us reformulate the definition of Springer
basic sets in this favorable situation.
The order $\leq^\NC_{{\mathfrak{B}}_{\bullet}}$, in this case, is
the order $\leq^\NC_{\mathrm{triv}}$ defined by
\[
E_{1} \leq^\NC_{\mathrm{triv}} E_{2} \Longleftrightarrow\left( E_{1}
= E_{2} \text{ or } \OC_{2} < \OC_{1}\right)
\]
where $E_{i} \in\Irr \KM W$ and $(\OC_{i},\LC%
_{i}) := \Psi_\KM(E_{i})$, for $i = 1,\ 2$.

The injection $\beta^\NC_{{\mathfrak{B}}_{\bullet}}$ is the map
$\beta_{\ell}: \Irr \FM W \hookrightarrow\Irr \KM W$ defined
in the following way: to $F \in\Irr \FM W$, we associate the unique
$E := \beta_{\ell}(F) \in\Irr\KM W$ such that $\Psi_\KM(E) =
\Psi_\FM(F) \in\PG$.

Then the pair ${\mathfrak{B}}_{\ell}:= (\leq^\NC_{\mathrm{triv}%
},\beta_{\ell})$ is a basic set datum for $W$. Moreover, any choice of a
family ${\mathfrak{B}}_{\bullet}$ of basic set data for the $A_{G}%
(\OC)$ will give rise to a basic set datum for $W$ with the same
injection $\beta_{\ell}$, since the trivial order $=_\OC$ is
contained in any chosen order $\leq_\OC$, and thus the order
$\leq^\NC_{\mathrm{triv}}$ is contained in $\leq^{\mathcal{N}%
}_{{\mathfrak{B}}_{\bullet}}$, and Proposition \ref{prop:comp} applies.

\begin{remark}
Basic sets for Weyl groups (and actually for Hecke algebras) in the good prime
case have already been constructed, for example in \cite{Geck-Rouquier}. For
symmetric groups, this goes back to James \cite{James:irr}.
\end{remark}

\subsection{The case of the general linear group}

Although the modular Springer correspondence for ${\mathbf{GL}}_{n}$ has been
determined in \cite{Juteau:thesis}, we will do it again here in the language
of Springer basic sets, as this case is much simpler than the other classical
types, for which it will serve as a model.

Let $\Part_{n}$ denote the set of all partitions of $n$. The prime $\ell$
being fixed, the subset $\Part^{(\ell)}_{n}$ consists of $\ell$-regular
partitions, that is, those which do not contain entries repeated at least
$\ell$ times. The symbols $\lambda\vdash n$ and $\lambda\vdash^{(\ell)}
n$ mean respectively $\lambda\in\Part_{n}$ and $\lambda\in\Part^{(\ell)}_{n}$.

Nilpotent orbits for ${\mathbf{GL}}_{n}$ are parametrized by $\Part_{n}$ via the Jordan normal form. We will denote them by $\OC_\lambda$, for $\lambda\in\Part_n$. The orbit closure order is given by the classical
dominance order on partitions.

It turns out that the simple modules of $\KM{\mathfrak{S}}_{n}$, called Specht modules, are also parametrized by 
$\Part_{n}$. 
In \cite{James:irr}, James classifies the simple modules for ${\mathbb{F}}{\mathfrak{S}}_{n}$ as follows. The Specht 
module $S^{\lambda}$ is defined as a certain submodule of the permutation module $M^{\lambda} = 
\KM[\SG_n/\SG_\lambda]$, 
where ${\mathfrak{S}}_{\lambda}$ is the parabolic subgroup corresponding to the partition $\lambda$. The 
permutation module $M^{\lambda}$ is endowed with a natural scalar product, which restricts to a scalar product on 
$S^{\lambda}$. The module and the bilinear form are defined over $\OM$ (actually over ${\mathbb{Z}}$), and thus one can 
reduce them to get a module for $\FM {\mathfrak{S}}_{n}$, still called a Specht module, endowed with a bilinear form, which 
no longer needs to be non-degenerate. Then the quotient of the Specht module $S^{\lambda}$ by the radical of the form is 
either zero or a simple module denoted by $D^{\lambda}$. The partitions giving a non-zero result are exactly the 
$\ell$-regular 
partitions. The $D^{\lambda}$, for $\lambda\in\Part^{(\ell)}_{n}$, form a complete set of representatives of 
$\FM{\mathfrak{S}}_{n}$-modules. 
Moreover, James shows that (in our language) the pair consisting of the dominance order 
on partitions, and the
injection sending $D^{\lambda}$ ($\lambda\in\Part^{(\ell)}_{n}$) to $S^{\lambda}$, is a basic set datum, which we will call the 
James basic set datum.

Now, the ordinary Springer correspondence sends the Specht module $S^{\lambda}$ to the orbit $\OC_{\lambda^{*}}$ (with the trivial local system), where $\lambda^{*}$ denotes the conjugate partition of $\lambda$. It follows that the Springer and James basic set data involve the same order relation, hence they coincide by Corollary
\ref{cor:unique beta}. We have deduced the modular Springer correspondence for
${\mathbf{GL}}_{n}$.

\begin{theorem}
\label{th:GLn} 
For $\mu \in \Part^{(\ell)}_{n}$, we have $\Psi^{\GL_n}_\FM(D^\mu) = (\OC_{\mu^*}, \FM)$.
\end{theorem}

To simplify the notation, we set $d_{\lambda,\mu}^{{\mathfrak{S}}_{n}}:=d_{S^{\lambda},D^{\mu}}^{{\mathfrak{S}}_{n}}$ and 
$d_{\lambda,\mu}^\NC:=d_{(\OC_{\lambda},\KM),(\OC_{\mu},\FM)}^\NC$. Theorems \ref{th:dec} and \ref{th:GLn} imply that
\[
d_{\lambda,\mu}^{{\mathfrak{S}}_{n}}=d_{\lambda^{\ast},\mu^{\ast}%
}^\NC%
\]
for $\lambda\vdash n$ and $\mu\vdash^{(\ell)}n$.

\section{Dipper--James basic set data}

\label{sec:DJ}

\subsection{Type \texorpdfstring{$B_{n}/C_{n}$}{B/C}}
\label{subsec:DJ-BC}

Let $W_{n}$ be the Weyl group of type $B_{n}$. We can identify it with the group of signed permutations of the set $\{\pm 1, \dots, \pm n\}$. The classification of ordinary and modular simple modules is completely analogous to the case of the 
symmetric group \cite{Dipper-James:B_n}. We have a parametrization $\Irr \KM W_{n} = \{{\mathbf{S}}^{\boldsymbol{\lambda
}} \mid\lamb\in\Bipart_{n}\}$, where $\Bipart_{n}$ denotes the set of all bipartitions of $n$, that is, the set of pairs of partitions
$\lamb = (\lambda^{(1)},\lambda^{(2)})$ such that $|\lambda^{(1)}| + |\lambda^{(2)}| = n$. For example, the trivial 
representation is labeled by the pair $((n),\emptyset)$ and the sign representation by $((\emptyset,(1^{n}))$.

Again, those modules are defined over ${\mathbb{Z}}$ and endowed with a bilinear form; factoring out the radical of the form 
over $\FM$ yields $0$ or a simple module; and one obtains a complete collection of simple $\FM W_{n}$-modules in this 
way. In our case ($\ell\neq2$), the result is that we have a parametrization $\Irr \FM W_{n} = \{ \DB^\lamb 
\mid\lamb\in\Bipart^{(\ell)}_{n}\}$, where $\Bipart^{(\ell)}_{n}$ denotes the set of $\ell$-regular bipartitions of $n$, that is, the 
bipartitions $\lamb=(\lambda^{(1)},\lambda^{(2)})$ such that $\lambda^{(1)}$ and $\lambda^{(2)}$ are $\ell$-regular. Naming 
the modular simple modules in this way implicitly amounts to define an injection $\gamma_{\ell}$ from $\Irr \FM W_{n}$ to 
$\Irr 
\KM W_{n}$, sending ${\mathbf{D}}^\lamb$ to ${\mathbf{S}} ^\lamb$, for $\lamb\in\Bipart^{(\ell)}_{n}$.

By \cite{Dipper-James:B_n}, a Morita equivalence allows to express the
decomposition numbers $d^{W_{n}}_{\lamb,{\boldsymbol{\mu}}%
}:=d^{W_{n}}_{{\mathbf{S}}^\lamb,{\mathbf{D}}%
^{\boldsymbol{\mu}}}$ of $W_{n}$ in terms of decomposition numbers for the
symmetric groups ${\mathfrak{S}}_{r}$ for $0\leq r\leq n$.

\begin{proposition}
\cite{Dipper-James:B_n}
(Recall that $\ell\neq 2$.)  Let $\lamb=(\lambda^{(1)},\lambda^{(2)})$ and ${\boldsymbol{\mu}}=(\mu^{(1)},\mu^{(2)})$ be two
bipartitions of $n$, the latter being $\ell$-regular. If
$|\lambda^{(1)}|=|\mu^{(1)}| =:r$, then
\[
d^{W_{n}}_{\lamb,{\boldsymbol{\mu}}} = d^{{\mathfrak{S}}_{r}%
}_{\lambda^{(1)},\mu^{(1)}} \cdot d^{{\mathfrak{S}}_{n-r}}_{\lambda^{(2)}%
,\mu^{(2)}}%
\]
and otherwise we have $d^{W_{n}}_{\lamb,{\boldsymbol{\mu}}%
}=0$.
\end{proposition}

This is good motivation to introduce the following very coarse order relation.

\begin{definition}[Dipper--James order on bipartitions]
\label{def:DJ}
Let $\lamb=(\lambda
^{(1)},\lambda^{(2)})$ and ${\boldsymbol{\mu}}=(\mu^{(1)},\mu^{(2)})$ be two
bipartitions of $n$. We say that $\lamb\leq_{DJ}%
{\boldsymbol{\mu}}$ if, for $i\in\{1,2\}$, we have $|\lambda^{(i)}|=|\mu
^{(i)}|$ and $\lambda^{(i)}\leq\mu^{(i)}$. This induces a partial order on
$\Irr \KM W_{n}$ that we still denote by $\leq_{DJ}$.
\end{definition}

By the results for the symmetric group, the preceding proposition implies:

\begin{proposition}
\label{prop:DJbasicsetdatum} The pair $(\leq_{DJ},\gamma_{\ell})$ is a basic
set datum for $W_{n}$.
\end{proposition}

We will call it the Dipper-James basic set datum.

Since the two versions of the ordinary Springer correspondence are related by
tensoring with the sign character $\Sgn_{W}$, it is useful to recall how this
translates combinatorially. If $\lamb=(\lambda^{(1)}%
,\lambda^{(2)})$ is a bipartition, we set $\lamb%
^{*}:=({\lambda^{(2)}}^{*},{\lambda^{(1)}}^{*})$.

\begin{proposition}
[\cite{Geck-Pfeiffer}, Theorem 5.5.6]For $\lamb$ in
$\Bipart_{n}$, we have $\Sgn_{W_{n}}\otimes{\mathbf{S}}^{\boldsymbol{\lambda}%
}={\mathbf{S}}^{\lamb^{*}}$.
\end{proposition}

\begin{remark}
We have $\lamb\leq_{DJ}{\boldsymbol{\mu}}\Leftrightarrow
{\boldsymbol{\mu}}^{*}\leq_{DJ} \lamb^{*}$.
\end{remark}

\subsection{Type \texorpdfstring{$D_{n}$}{D}}
\label{subsec:DJ-D}

For $n \geq 2$, let $W'_{n}$ be the index two subgroup of $W_{n}$ consisting of signed permutations with an even number of sign flips. It is a Weyl group of type $D_{n}$.

A classification of the simple $\KM W^{\prime}_{n}$-modules is obtained as follows:

\begin{itemize}
\item if $(\lambda^{(1)}, \lambda^{(2)}) \in \Bipart_{n}$ satisfies $\lambda^{(1)}\neq\lambda^{(2)}$, then the $\KM W_{n}$-modules ${\mathbf{S}}^{(\lambda
^{(1)},\lambda^{(2)})}$ and ${\mathbf{S}}^{(\lambda^{(2)},\lambda^{(1)})}$
have the same restriction to $W^{\prime}_{n}$, a simple
$\KM W^{\prime}_{n}$-module which we will denote by ${\mathbf{S}}^{[\lambda^{(1)},\lambda^{(2)}]}$;

\item the restriction of ${\mathbf{S}}^{(\lambda,\lambda)}$ splits into a
direct sum of two non-isomorphic simple $\KM W^{\prime}_{n}$-modules
which are denoted by ${\mathbf{S}}^{[\lambda,+]}$ and ${\mathbf{S}}%
^{[\lambda,-]}$ (this case arises only for even $n$).
\end{itemize}

Moreover, every simple $\KM W^{\prime}_{n}$-module arises exactly
once in this way. Note that we use the notation $(a,b)$ for an ordered pair,
and $[a,b]$ for an unordered pair. Let $\Unbipart_{n}$ denote the set of
unordered bipartitions of $n$. Thus $\Irr \KM W^{\prime}_{n}$ is parametrized by
\[
\Dpart_{n} =
\begin{cases}
\Unbipart_{n} & \text{if $n$ is odd},\\
\{ [\lambda_{1},\lambda_{2}] \in\Unbipart_{n} \mid\lambda_{1} \neq\lambda_{2}
\} \cup\{ [\lambda,\pm] \mid\lambda\in\Part_{\frac n 2} \} & \text{if $n$ is
even}.
\end{cases}
\]
We denote by $\Unbipart_{n}^{(\ell)} \subset \Unbipart_{n}$ the subset of $\ell$-regular
unordered bipartitions (those where each partition is $\ell$-regular), and
similarly for $\Dpart_{n}^{(\ell)}$.

For any $E\in\Irr \KM W^{\prime}_{n}$ we denote by $a_{E}$ the
$a$-invariant of $E$ (see Definitions 6.5.7 and 9.4.8 in \cite{Geck-Pfeiffer}).
For any $F\in\Irr \FM W^{\prime}_{n}$, we define
\[
a_{F}:=\min\{a_{E} \mid E\in\Irr \KM W^{\prime}_{n}\text{ and
}d^{W^{\prime}_{n}}_{E,F}\neq0\}.
\]

\begin{proposition}
(Relation between $\Irr \FM W_{n}$ and $\Irr \FM W^{\prime
}_{n}$ \cite[\S 6]{Geck:extended}.)
\label{condainv}

\begin{enumerate}
\item \label{propa-invariant} There is a parametrization
\[
\Irr \FM W^{\prime}_{n} = \{{\mathbf{E}}^\lamb
\mid\lamb \in\Dpart_{n}^{(\ell)} \}
\]
such that, if $\gamma_{\ell}%^{\prime}
:\Irr\FM W^{\prime}_{n}\rightarrow\Irr\KM W^{\prime}_{n}$ is induced by the natural 
inclusion $\Dpart_{n}^{(\ell)} \rightarrow\Dpart_{n}$, the following properties hold:

\begin{itemize}
\item for all $F\in\Irr\FM W^{\prime}_{n}$, we have $d^{W^{\prime
}_{n}}_{\gamma_{\ell}(F),F}=1$ and $a_{\gamma_{\ell}(F)} =
a_{F}$;

\item given $E\in\Irr \KM W^{\prime}_{n}$ and $F\in\Irr \FM W^{\prime}_{n}$, we have
\[
d^{W^{\prime}_{n}}_{E,F} \neq 0 \Rightarrow a_{F} < a_{E} \text{ or } E = \gamma_{\ell}
(F).
\]
\end{itemize}

\item Moreover, we have the following relations between the simple modules for
$\FM W_{n}$ and for $\FM W^{\prime}_{n}$:

\begin{itemize}
\item if $\lambda^{(1)}\neq\lambda^{(2)}$, then ${\mathbf{D}}^{(\lambda
^{(1)},\lambda^{(2)})}$ and ${\mathbf{D}}^{(\lambda^{(2)},\lambda^{(1)})}$
have the same restriction to $W^{\prime}_{n}$ and this restriction is the
simple $\FM W^{\prime}_{n}$-module ${\mathbf{E}}^{[\lambda
^{(1)},\lambda^{(2)}]}$;

\item the restriction of ${\mathbf{D}}^{(\lambda,\lambda)}$ splits into a
direct sum of two non-isomorphic simple $\FM W^{\prime}_{n}$-modules
which are ${\mathbf{E}}^{[\lambda,+]}$ and ${\mathbf{E}}^{[\lambda,-]}$.
\end{itemize}
\end{enumerate}
\end{proposition}

\begin{definition}
[Dipper--James order on unordered bipartitions]We define a partial order on the
set $\Unbipart_{n}$ of unordered bipartitions of $n$ by:
\[%
\begin{array}
[c]{c}%
[\lambda^{(1)},\lambda^{(2)}]\leq_{DJ}[\mu^{(1)},\mu^{(2)}]\\
\Updownarrow\\
(\lambda^{(1)},\lambda^{(2)})\leq_{DJ}(\mu^{(1)},\mu^{(2)})\text{ or }%
(\lambda^{(1)},\lambda^{(2)})\leq_{DJ}(\mu^{(2)},\mu^{(1)})
\end{array}
\]
(It is easy to check that this is indeed a partial order.)

We have a natural projection $\varphi: \Dpart_{n} \rightarrow\Unbipart_{n}$,
sending $[\lambda,\mu]$ to $[\lambda,\mu]$ and $[\lambda,\pm]$ to $[\lambda,\lambda]$. The order $\leq_{DJ}$ on
$\Unbipart_{n}$ induces an order that we still denote by $\leq_{DJ}$ on
$\Dpart_{n}$, hence on $\Irr \KM W^{\prime}_{n}$, such that
$\lamb \ < {\boldsymbol{\mu}}$ if and only if $\varphi
(\lamb) < \varphi({\boldsymbol{\mu}})$ for
$\lamb, {\boldsymbol{\mu}} \in\Dpart_{n}$. So the irreducible
modules ${\mathbf{S}}^{[\lambda,+]}$ and ${\mathbf{S}}^{[\lambda,-]}$ are not comparable.
\end{definition}

\begin{proposition}
\label{prop:DJbasicsetdatumtypeD} The pair $(\leq_{DJ},\gamma_{\ell}
)
$ is a basic set datum for $W^{\prime}_{n}$.
\end{proposition}

To prove the proposition, we will need the fact that the restriction from
$W_{n}$ to $W^{\prime}_{n}$ commutes with decomposition maps. Let us introduce
some notation. If $A$ is a finite dimensional algebra over a field, we denote
by $R_{0}(A)$ the Grothendieck group of the category of finite dimensional
$A$-modules. The class of an $A$-module $V$ in $R_{0}(A)$ is denoted by $[V]$.

\begin{lemma}
[\cite{Geck:extended}, lemma 5.2]The restriction of modules from
$\Irr \KM W_{n}$ to $\Irr \KM W^{\prime}_{n}$ (resp. from
$\Irr \FM W_{n}$ to $\Irr \FM W^{\prime}_{n}$) induces maps
between Grothendieck groups fitting into the following commutative diagram
\[
\begin{CD} R_0(\KM W_n) @>\Res>> R_0(\KM W'_n)\\ @VVdV @VVd'V\\ R_0(\FM W_n) @>\Res>> R_0(\FM W'_n) 
\end{CD}
\]
where $d$ and $d^{\prime}$ denote the decomposition maps.
\end{lemma}

\begin{proof}
[Proof of Proposition \ref{prop:DJbasicsetdatumtypeD}]By the lemma, for
$\lambda^{(1)}\neq\lambda^{(2)}$, we have
\[%
\begin{array}
[c]{rcl}%
d^{\prime}([{\mathbf{S}}^{[\lambda^{(1)},\lambda^{(2)}]}]) & = &
\Res(d([{\mathbf{S}}^{(\lambda^{(1)},\lambda^{(2)})}]))\\
& = & \Res([{\mathbf{D}}^{(\lambda^{(1)},\lambda^{(2)})}])+\Res( \text{larger
terms for}\leq_{DJ})\\
& = & [{\mathbf{E}}^{[\lambda^{(1)},\lambda^{(2)}]}]+ (\text{larger terms
for}\leq_{DJ})
\end{array}
\]
For $\lambda^{(1)}= \lambda^{(2)} = \lambda$, we have
\[
d^{\prime}([{\mathbf{S}}^{[\lambda,+]}]+ [{\mathbf{S}}^{[\lambda
,-]}])=[{\mathbf{E}}^{[\lambda,+]}]+ [{\mathbf{E}}^{[\lambda,-]}]+
(\text{larger terms for}\leq_{DJ})
\]

By Proposition \ref{condainv} \eqref{propa-invariant}, we get $d^{W^{\prime
}_{n}}_{{\mathbf{S}}^{[\lambda,+]},{\mathbf{E}}^{[\lambda,-]}}=0$ as
${\mathbf{S}}^{[\lambda,-]}$ and ${\mathbf{S}}^{[\lambda,+]}$ have the same
$a$-invariant.

This proves that $(\leq_{DJ}, \gamma_{\ell}
)$ is a basic set datum.
\end{proof}

We define the transposition $*$ on $\widetilde\Bipart_{n}$ by $[\lambda^{(1)}, \lambda^{(2)}]^{*} = 
[{\lambda^{(2)*}},{\lambda^{(1)*}}]$, 
and $[ \lambda, \pm]^{*} = [\lambda^{*}, \pm]$.

\begin{proposition}
[\cite{Geck-Pfeiffer}, Remark 5.6.5]
For $\lamb\in\widetilde\Bipart_{n}$, we have $\Sgn_{W'_n}\otimes{\mathbf{S}}^{\lamb} = {\mathbf{S}}^{\lamb^{*}}$.
\end{proposition}

\begin{remark}
For $\lamb,\mub \in \widetilde\Bipart_{n}$, we have
$\lamb \leq_{DJ} \mub \Leftrightarrow \mub^* \leq_{DJ} \lamb^{*}$.
\end{remark}

\section{Symbols and dimensions of Springer fibers}
\label{sec:symbols}

From now on, $G$ is a classical group defined over $\overline\FM_{p}$, with $p$ odd. In this section, we recall the bare 
minimum that we need from the combinatorics of the ordinary Springer correspondence for classical types, in order to prove 
our main theorem saying that the modular Springer correspondence can be deduced by restricting to $\ell$-regular 
bipartitions. Namely, we introduce symbols, and describe the map $\Phi^{G}_{\co}$ from bipartitions to symbols (which is very easy). To complete the description, one would then have to explain the correspondence between symbols and local systems on nilpotent orbits. However, we only need the combinatorial formula computing the dimension of the corresponding Springer fibre, directly from the symbol, in order to prove our main theorem. We will come back to the complete description in \S \ref{sec:comb}.

Since we only study the non-generalized case, the isogeny type is irrelevant to the determination of the correspondence, and 
we may assume that $G$ is a symplectic or special orthogonal group. We work with the Lie algebra ${\mathfrak{g}}$, but 
since $p$ is good for $G$, it gives the same results as for the group.

\subsection{Generalized symbols}

Following \cite{Lusztig-Spaltenstein:classical-Springer, Geck-Malle:support},
we will use generalized symbols. Let $n$, $r$ and $s$ be positive integers,
and $d\in\{0,1\}$. We denote by $X^{r,s}_{n,d}$ the set of ordered pairs
$\tbinom A B$ of finite sequences $A = (a_{1}, \dots, a_{m + d})$, $B =
(b_{1}, \dots, b_{m})$, such that:

\begin{enumerate}
\item $a_{i} - a_{i - 1} \geq r + s \text{ for } 1 < i \leq m + d$,

\item $b_{i} - b_{i - 1} \geq r + s \text{ for } 1 < i \leq m$,

\item $b_{1} \geq s$,

\item $\sum_{i = 1}^{m + d} a_{i} + \sum_{i = 1}^{m} b_{i} = n + rm(m + d - 1) + sm(m + d)$,
\end{enumerate}
taken modulo the shift operation
\[
\tbinom A B \mapsto\tbinom{0,\ a_{1} + r + s,\ \dots,\ a_{m + d} + r + s}{s,\ b_{1} + r + s,\ \dots,\ b_{m} + r + s }.
\]

If $s = d = 0$, we denote by $Y^{r}_{n}$ the quotient of $X^{r,0}_{n,0}$ by the operation of swapping the two rows of a 
symbol. The class of $\tbinom A B$ will be denoted $\genfrac{[}{]}{0pt}{1}{A}{B}$. Finally, let $\widetilde Y^{r}_{n}$ be the set 
similar to $Y^{r}_{n}$, but where each symbol with identical rows has a $\pm$ decoration.

Varying $r$ and $s$, these symbols are useful for parametrizing both irreducible representations of classical Weyl groups, 
and local systems on nilpotent orbits, for the different classical types. It is remarkable that all instances of the Springer 
correspondence (generalized case, bad prime $p = 2$\dots) for classical types can be expressed via generalized symbols. 
The integer $n$ will be the rank, and here the ``defect'' $d$ is limited to $\{0,1\}$ because we only consider the 
non-generalized 
case.

Note that the set $X^{r,s}_{0,d}$ consists of a single element,
\[
\Lambda^{r,s}_{0,d} =
\begin{cases}
\tbinom\emptyset\emptyset= \tbinom0 s = \tbinom{0\quad r + s}{s\quad r + 2s} =
\tbinom{0\quad r + s\quad2r + 2s}{s\quad r + 2s \quad2r + 3s} = \cdots &
\text{if } d = 0,\\[.5em]
\tbinom0 \emptyset= \tbinom{0 \quad r + s}{ s \quad} = \tbinom{0 \quad r + s
\quad2r + 2s}{s \quad r + 2s \quad} = \cdots & \text{if } d = 1.
\end{cases}
\]
Adding representatives of the same length componentwise defines an addition
\[
X^{r,s}_{n,d} \times X^{r^{\prime},s^{\prime}}_{n^{\prime},d}
\longrightarrow X^{r+r^{\prime},s+s^{\prime}}_{n+n^{\prime},d}.
\]
We define the \emph{reading} of a symbol $\Lambda= \tbinom A B$ by the
sequence $w(\Lambda) := (a_{1}, b_{1}, a_{2}, b_{2}, a_{3}\ldots)$ obtained by
alternating the entries of $A$ and $B$. A symbol is called
\emph{distinguished} if the entries of its reading weakly increase from left to right. 
Two symbols are \emph{similar} if the multisets of entries of suitable representatives coincide. Any symbol $\Lambda$ is similar to a (clearly unique) distinguished symbol $\dist(\Lambda)$: indeed, one can keep exchanging pairs of consecutive entries which are in the wrong order, until one reaches a distinguished symbol \cite[\S 2.A]{Geck-Malle:support}. Similar symbols will correspond to different local systems on the same nilpotent orbit.

\subsection{From bipartitions to symbols}
\label{subsec:bipart symb}

Recall that the irreducible characters of the Weyl group $W_{n}$ of type $B_{n}$ or $C_{n}$ are parametrized by the set
$\Bipart_{n}$ of bipartitions of $n$; the latter may be identified with $X^{0,0}_{n,1}$ by writing entries of each partition in 
weakly increasing order, and adding $0$ entries if necessary to get symbols of defect $1$. Similarly, for type $D_{n}$, we 
identify $\Irr \KM W'_{n}$ with $\widetilde\Bipart_{n} = \widetilde Y^{0}_{n}$ (keeping track of the decorations).
It will be convenient to let $\Bipart_W$ denote either $\Bipart_n$ if $W$ is of type $B_n/C_n$, or $\Dpart_n$ if $W$ is of type 
$D_n$. 

The pairs in $\PG_{G,\KM}^\Spr = \im\Phi^G_\KM$ can be parametrized by a set of symbols that we will denote by $\Symb_{G}^\Spr$ (see \S \ref{sec:comb}); it is $X_{n,1}^{1,1}$ for type $C_{n}$, $X_{n,1}^{2,0}$ for type $B_{n}$, and 
$\widetilde Y_{n}^{2}$ for type $D_{n}$.
Then $\Phi^{G}_{\KM} = \Psi^{G}_{\KM} \circ (\Sgn_W \otimes - )$ is deduced from the map 
$\Phi^G_\co 
: \Bipart_{W} \to \Symb_{G}^{\Spr}$ defined as follows:
\begin{align}
\text{Type }C_{n}: &  \qquad \Bipart_{n} = X_{n,1}^{0,0}\longrightarrow X_{n,1}^{1,1}, &
\tbinom{\alpha}{\beta} &  \longmapsto \tbinom{\alpha}{\beta} + \Lambda_{0,1}^{1,1}; \\
\text{Type }B_{n}: &  \qquad \Bipart_{n} = X_{n,1}^{0,0}\longrightarrow X_{n,1}^{2,0}, &
\tbinom{\alpha}{\beta} &  \longmapsto \tbinom{\alpha}{\beta} + \Lambda_{0,1}^{2,0}; \\
\text{Type }D_{n}: &  \qquad \widetilde \Bipart_{n} = \widetilde Y_{n}^{0}\longrightarrow \widetilde Y_{n}^{2}, &
\genfrac{[}{]}{0pt}{1}{\alpha}\beta & \longmapsto \genfrac{[}{]}{0pt}{1}{\alpha}{\beta} + \Lambda_{0,0}^{2,0}.
\end{align}
(keeping track of the decorations for degenerate bipartitions or symbols in type $D_{n}$ with $n$ even), so that characters 
sent to similar symbols have the same associated nilpotent orbit (but with different local systems).
Let us call $\Lambda_{G}$ the symbol that is added in the description of $\Phi^G_\co$. Note that $\Phi^G_\co$ describes 
$\Phi^{G}_{\KM}$ 
rather than $\Psi^{G}_{\KM}$, whereas the Springer basic set datum is defined with respect to the latter. But 
tensoring by the sign character amounts to transposing bipartitions, which just reverses the order on bipartitions.

\subsection{Dimensions of Springer fibres}
For $(\OC,\LC)\in \PG_{G}$, let us denote $\delta((\OC,\LC))=\dim \BC_{\OC}$. 
There is a combinatorially defined 
function $\delta_\co : \Symb_{G} \to \NM$ describing $\delta$ (see \cite[\S 2.E]{Geck-Malle:support}). Since this 
quantity depends only on the nilpotent orbit and not on the local system, the function $\delta_\co$ needs to be constant on 
similarity classes. For $\Lambda\in \Symb_G$, we have
\[
\delta_\co (\Lambda) = \sigma( w( \Lambda ) ) - \sigma( w (\Lambda_G ) ),
\text{ where } \sigma( s_{1},\dots, s_{2m+d} ) = \sum_{1\leq i < j \leq 2m + d} \min( s_i, s_j ).
\]
Of course one should take representatives of $\Lambda$ and $\Lambda_{G}$ of the same size: here $m$ is the length of the 
second line. We can be more explicit: for symbols in $\Symb_G^\Spr$ (i.e. $d=0,1$), we have
\[
\sigma( w (\Lambda_G ) ) =
\begin{cases}
m(4m^{2} - 1)/3 & \text{ in type $C_{n}$};\\
m(m - 1)(4m + 1)/3 & \text{ in type $B_{n}$};\\
m(m - 1)(4m - 5)/3 & \text{ in type $D_{n}$}.
\end{cases}
\]

\begin{definition}
Let $\leq_\delta$ be the order relation on $\Irr \KM W = \Bipart_{W}$ defined by
\[
\bfl <_{\delta} \bfm \Longleftrightarrow \delta\Big( \Psi_{\KM}(\bfl) \Big)<  \delta\Big( \Psi_{\KM}(\bfm) \Big)\Longleftrightarrow \delta_\co\Big( \Phi^G_\co(\bfl^*) \Big) < \delta_\co\Big( \Phi^G_\co(\bfm^*) \Big).
\]
\end{definition}
Note that we have used the fact that the two versions of the Springer correspondence differ by tensoring by the 
sign character, which amounts to transposing the bipartitions before applying the correspondence.

Let $\bfl,\bfm\in\Bipart_{W}$, and set $\Psi_{\KM}( \bfl ) = (\OC,\LC_{\rho})$ and $\Psi_{\KM}( \bfm ) = (\OC',\LC_{\rho'})$. 
Note that
\[
\bfl <_\triv^{\NC} \bfm \Longleftrightarrow \OC > \OC' \Longrightarrow
\dim \OC > \dim \OC' \Longleftrightarrow \dim \BC_{\OC} < \dim \BC_{\OC'} \Longleftrightarrow
\bfl <_{\delta} \bfm,
\]
where we have used the fact that $\dim \BC_{\OC}$ is equal to half 
the codimension of $\OC$ in $\NC$ \cite{Springer:Trig,Steinberg}.
Thus we have $\leq^{\NC}_{\triv} {}\subset{} \leq_{\delta}$; since $(\leq^{\NC}_{\triv},\beta_{\ell})$ is a basic set datum, 
we deduce that thus $(\leq_\delta, \beta_{\ell})$ is a basic set datum as well, by Proposition \ref{prop:comp}. Remarkably, this 
very limited part of the combinatorics is enough for the proof of our main theorem in the next section.

\section{Modular Springer correspondence for classical types}
\label{sec:mod}

\begin{proposition} \label{prop:DJ delta}
We have $\leq_{DJ} {}\subset{} \leq_{\delta}$.
\end{proposition}

\begin{proof}
Taking into account the fact that transposing the bipartitions reverses the Dipper--James order on $\Bipart_{W}$, 
we need to show that for $\bfl$ and $\bfm$ in $\Bipart_W$, we have
\[
\bfl <_{DJ} \bfm \Longrightarrow \delta_\co \Big( \Phi_\co^G(\bfm) \Big) < \delta_\co \Big( \Phi^G_\co(\bfl) \Big).
\]

The proof is similar for symplectic and special orthogonal groups. We give the proof for $G=\Sp_{2n}$.
Moreover, it is sufficient to prove the result when $\bfm$ covers $\bfl$, in which case we may assume for example that 
$\lambda^{(2)} 
= \mu^{(2)}$, while $\lambda^{(1)}$ is obtained from $\mu^{(1)}$ by increasing one part by $1$ and decreasing 
another part by $1$ (the case where  $\lambda^{(1)} = \mu^{(1)}$ is similar). Then the symbols $\Phi^G_\co(\bfl)$ and 
$\Phi^G_\co(\bfm)$ 
have the same relationship: their second lines are identical, while the first line of $\Phi^G_\co(\bfl)$ is 
obtained from the first line of $\Phi^G_\co(\bfm)$ by increasing one part by $1$ and decreasing another part by $1$.

Let us reorder the entries of $\Phi^G_\co(\bfl)$ and $\Phi^G_\co(\bfm)$ into the respective sequences $x_{1}\leq x_{2} \leq 
\dots \leq x_{2m + 1}$ and $y_{1} \leq y_{2} \leq \dots \leq  y_{2m + 1}$. Then there exist indices $a < b$ such that $y_{b} = x_{b} + 
1$ and $y_{a} = x_{a} - 1$, while $y_{i} = x_{i}$ for all $i \neq a, b$. We get:

\[
\begin{array}{lll}
\delta_\co \Big( \Phi^G_\co(\bfm) \Big) - \delta_\co \Big( \Phi^G_\co(\bfl) \Big)
&=&\sum\limits_{1\leq i<j\leq 2m + 1} \min(y_i,y_j)- \sum\limits_{0\leq i<j\leq 2m} \min(x_i,x_j)\\[1em]
&=&\sum\limits_{1 \leq i \leq 2m + 1}(2m-i)(y_i-x_i)\\[1em]
&=& (2m-a)(y_{a}-x_{a})+(2m-b)(y_{b}-x_{b})\\
&=& a - b < 0.
\end{array}
\]
\end{proof}

\begin{corollary}
The inclusion $\beta_\ell : \Irr \FM W \hookrightarrow \Irr \KM W$, which was introduced in \S \ref{subsec:good}, coincides 
with the inclusion $\gamma_{\ell}$ defined in \S \ref{subsec:DJ-BC} and \S \ref{subsec:DJ-D}.
\end{corollary}

\begin{proof}
Recall that $(\leq_{DJ}, \gamma_{\ell})$ is a basic set datum by Proposition 
\ref{prop:DJbasicsetdatum}. By Propositions \ref{prop:DJ delta} and \ref{prop:comp}, $(\leq_{\delta}, \gamma_{\ell})$ is a 
basic set datum as well. We have seen at the end of \S \ref{sec:symbols} that $(\leq_\delta, \beta_{\ell})$ is also a basic 
set datum. It follows from Corollary \ref{cor:unique beta} that $\beta_{\ell} = \gamma_{\ell}$.
%The same argument applies to $D_{n}$, using $\gamma^{}_{\ell}$ and Proposition \ref{prop:DJbasicsetdatumtypeD}.
\end{proof}

Reformulating, we get the main result of the paper.

\begin{theorem}
Let $p$ be an odd prime number, and let $G$ be a classical group over $\ov\FM_{p}$, with Weyl group $W$. For an odd 
prime $\ell\neq p$, the modular Springer correspondence for $G$ is described as follows: for $\lamb\in\Bipart_{W}^{(\ell)}$,
\[
\Psi^{G}_\FM(\DB^\lamb) = \Psi^{G}_{\KM}(\SB^\lamb) = \Phi^{G}_{\KM}(\SB^{\lamb^*}).
\]
\end{theorem}
For the statement to make sense, recall that we may identify $\PG_{\KM}$ with $\PG_{\FM}$ because $\ell \neq 2$.

\section{Compatibility between Dipper--James and Springer orders}
\label{sec:comb}

\subsection{Classical nilpotent orbits}

For $G$ a classical group with natural representation $V$, with $p$ good, a nilpotent orbit is determined by the Jordan type of its elements acting on $V$, with the exception of the so-called ``very even'' case in type $D_{n}$. More precisely, we have a set of partitions $\PC(G)$ (with decorations for the very even case), described below according to the type, such that the nilpotent orbits can be naturally called $\OC^G_\lambda$, $\lambda\in\PC(G)$:

\begin{enumerate}
\item $\mathcal{P}(C_{n})$ is the set of partitions of $2n$ in which each odd part appears with an even multiplicity (we 
ensure that the total number of parts is odd by adding a $0$ entry if necessary: this will be useful later);

\item $\mathcal{P}(B_{n})$ is the set of partitions of $2n+1$ in which each even part appears with an even multiplicity (the 
total number of parts is necessarily odd);

\item $\mathcal{P}(D_{n})$ is the set of partitions of $2n$ in which each even part appears with an even multiplicity, except 
that a partition with no odd part is decorated with a $\pm$ sign (the total number of parts is necessarily even).
\end{enumerate}

\subsection{From nilpotent orbits to bipartitions and symbols}

Let $\lambda\in\PC(G)$. Partition the sequence $\lambda=(\lambda
_{1}\leq\ldots\leq\lambda_{t})$ into blocks of lengths $1$ or $2$ such that if $\lambda$ is symplectic (resp. orthogonal) then all even (resp. odd) $\lambda_{i}$ lie in a block of length $1$, and all odd (resp. even) $\lambda_{i}$ lie in a block of length $2$. By \cite{Geck-Malle:support}, the symbol associated to $\lambda$ (i.e.~corresponding to the pair $(\OC_\lambda,\EM)$) is the unique distinguished symbol with reading $(c_{1},\dots,c_{t})$ given by
\[
\begin{cases}
c_{i} = \left\lfloor \lambda_{i} / 2 \right\rfloor + i - 1 & % \frac{\lambda_{i}}{2}
\text{if} \{\lambda_{i}\}\text{ is a block},\\
c_{i} = c_{i+1} = \left\lceil \lambda_{i} / 2 \right\rceil + i - 1 & % \frac{\lambda_{i}}{2}
\text{if }\{\lambda_{i},\lambda_{i+1}\}\text{ is a block}.
\end{cases}
\]
One obtains the corresponding bipartition (written as a symbol) by subtracting $\Lambda_{G}$. It has reading $(a_{1},\dots,a_{t})$, where in the symplectic case:
\[
\begin{cases}
\qquad\quad a_{i} = \lambda_{i} / 2 &
\text{if $\{\lambda_{i}\}$ is a block, }\\
(a_{i},a_{i+1}) = \left(  \left\lceil \lambda_{i} / 2 \right\rceil, \left\lfloor \lambda_{i} / 2 \right\rfloor \right) &
\text{if $\{\lambda_{i},\lambda_{i+1}\}$ is a block; }
\end{cases}
\]
and in the special orthogonal case:
\[
\begin{cases}
\qquad\quad a_{i} =  ( \lambda_{i} + (-1)^i ) / 2 & 
\text{if $\{\lambda_{i}\}$ is a block},
\\
( a_{i}, a_{i+1} )=\left(  \lambda_{i} / 2,  \lambda_{i} / 2 \right)   &
\text{if $\{\lambda_{i},\lambda_{i+1}\}$ is a block with $i$ odd},
\\
(a_{i},a_{i+1}) = \left(  \lambda_{i} / 2 + 1,  \lambda_{i} / 2 - 1 \right) &
\text{if $\{\lambda_{i},\lambda_{i+1}\}$ is a block with $i$ even.}
\end{cases}
\]
The resulting bipartition is actually the $2$-quotient of $\lambda$, as
defined in \cite{Mac}. Here we have explicit formulas because $\lambda
\in\mathcal{P}(B_{n})\cup\mathcal{P}(C_{n})\cup\mathcal{P}(D_{n})$ has a
special form.

In the following examples, we give a partition corresponding to a nilpotent orbit, then the corresponding bipartition, then the corresponding symbol. We highlight the nontrivial blocks. We have found it convenient to display the partition on two lines as well, in the same shape as the corresponding symbol.

\begin{example}
\label{exam_2quo}
Type $C_{28}$, $\lambda = (9^{2},7^{2},6^{2},4,3^{2},2)$. 
\[
\binom
	{0\quad \mathbf 3\quad  4\quad  6\quad  \mathbf 7\quad  \mathbf 9}
	{2\quad  \mathbf 3\quad  6\quad \mathbf 7\quad  \mathbf 9}
\longmapsto
\binom
	{0\quad \mathbf 2\quad 2\quad 3\quad \mathbf 3\quad \mathbf 4}
	{1\quad \mathbf 1\quad 3\quad \mathbf 4\quad \mathbf 5}
\longmapsto
\binom
	{0\quad \mathbf 4\quad 6\quad 9\quad \mathbf {11}\quad \mathbf {14}}
	{2\quad \mathbf 4\quad 8\quad \mathbf {11}\quad \mathbf {14}}
\]
\end{example}

\begin{example}
Type $B_{25}$, $\lambda= (8^{2}, 7, 6^{2}, 4^{2}, 3, 2^{2}, 1 )$.
\[
\binom
	{1\quad \mathbf 2\quad \mathbf 4\quad \mathbf 6\quad 7\quad \mathbf 8}
	{\mathbf 2\quad 3\quad \mathbf4\quad \mathbf 6\quad \mathbf 8}
\longmapsto
\binom
	{0\quad \mathbf 0\quad \mathbf 2\quad \mathbf 3\quad 3\quad \mathbf 3}
	{\mathbf 2\quad 2\quad \mathbf 2\quad \mathbf 3\quad \mathbf 5}
\longmapsto
\binom
	{0\quad \mathbf 2\quad \mathbf 6\quad \mathbf 9\quad 11\quad \mathbf {13}}
	{\mathbf 3\quad 5\quad \mathbf 7\quad \mathbf {10}\quad \mathbf {14}}
\]
\end{example}

\begin{example}
Type $D_{25}$, $\lambda = ( 8^{2}, 7, 6^{2}, 4^{2}, 3, 2^{2} )$.
\[
\begin{bmatrix}
	\mathbf 2& 3& \mathbf 4& \mathbf 6& \mathbf 8\\
	\mathbf 2& \mathbf 4& \mathbf 6& 7& \mathbf 8
\end{bmatrix}
\longmapsto
\begin{bmatrix}
	\mathbf 1& 1& \mathbf 1& \mathbf 2& \mathbf 4\\
	\mathbf 1& \mathbf 3& \mathbf 4& 4& \mathbf 4
\end{bmatrix}
\longmapsto
\begin{bmatrix}
	\mathbf 1& 3& \mathbf 5& \mathbf 8& \mathbf {12}\\
	\mathbf 1& \mathbf 5& \mathbf 8& 10& \mathbf {12}
\end{bmatrix}
\]
\end{example}

We will not need the precise formulation of the correspondence for nontrivial local systems: the only thing we need to know is 
that local systems on the same orbit have similar symbols. So we will only give a summary of the description: a local system 
is given by a sequence of signs; starting from the symbol of the trivial local system on the same orbit, one defines some 
``intervals" inside the symbol, and one swaps the two lines of some of the intervals, according to those signs (see 
\cite{Lusztig:ICC}\footnote{Actually there is a mistake in the description of the generalized correspondence in type $C$ in 
\cite[Theorem 12.3]{Lusztig:ICC}, which was corrected in \cite{Shoji:unip} (see his Remark 5.8). As pointed out in the final 
remark of \cite[\S 2.B]{Geck-Malle:support}, this does not affect the description of the nongeneralized Springer 
correspondence.}). This may change the defect of the symbol. Only those with defect $1$ are in the image of the 
nongeneralized Springer correspondence in types $B_{n}$ or $C_{n}$,\footnote{Actually for type $B_n$, the ``right'' notion is that of symbol of defect $\pm 1$ up to swapping the two lines, but it does not really matter for us since we don't need to describe precisely the local systems in the correspondence.} and only those with defect $0$ in type $D_{n}$.

\subsection{From bipartitions and symbols to nilpotent orbits}
\label{sec:symbol->partition}

\label{conv}

Consider a bipartition $\bfl\in\Bipart_W$ corresponding to a distinguished symbol $\Lambda=\binom S T$ or $\tsqbinom S T$. Write $S = (s_1, \dots, s_{m+d})$ and $T = (t_1, \dots, t_m)$. Let us describe how to recover the partition $\lambda\in\PC(G)$ of the corresponding nilpotent orbit. Let $x = w(\Lambda)$ be the reading of $\Lambda$. Then
\[
\lambda = 2\Big(\tilde x - w(\Lambda_G)\Big),
\]
where $\tilde x$ is a modification of $x$, which we will now describe according to the type, and $\Lambda_G$ was defined in \S \ref{subsec:bipart symb}. 

\medskip

Type $C_n$. A \emph{distinguished} pair in $\Lambda$ is a pair $(s_{i},t_{i})$ with $s_{i}=t_{i}$ or a pair $(t_{i},s_{i+1})$ 
with $t_{i}=s_{i+1}$. Since the rows of the symbols strictly increase, the distinguished pairs are pairwise disjoint. Then $\tilde x$ is the sequence obtained from $x$ by changing each distinguished pair $(s_{i},s_{i})$ into
$(s_{i}-1/2,s_{i}+1/2)$.

\smallskip

Type $B_{n}$ or $D_{n}$. A \emph{distinguished} pair is a pair
$(t_{i},s_{i+1})$ such that $t_{i}=s_{i+1}$. A \emph{frozen} pair is a pair $(s_{i},t_{i})$ with $s_{i}=t_{i}$. Since the rows of the symbols strictly increase, distinguished and frozen pairs are disjoint. Then $\tilde x$ is the sequence obtained from $x$ by changing each $s_{i}$ into $s_{i}+1/2$ and each $t_{i}$ into $t_{i}-1/2$ provided they do not appear in a distinguished or frozen pair, and each distinguished pair $(t_{i},t_{i})$ into $(t_{i}-1,t_{i}+1)$, 

\medskip

One sees easily this procedure is an inverse to that of the preceding section, hence we have a bijection between distinguished symbols and nilpotent orbits, which was denoted by $\Pi^G_1$ in the diagram of \S \ref{subsec:diag}.

If we start with a non-distinguished symbol $\Lambda$, we first have to reorder its entries to get a distinguished symbol $\dist(\Lambda)$, and then apply the procedure above to get the corresponding nilpotent orbit. The local system is encoded in what happens during that sorting, but we do not need a precise description to obtain the main result of this section: local systems are not comparable for our Springer order $\leq^\NC_\triv$ anyway (this is possible only because $\ell \neq 2$).

\subsection{Compatibility between Dipper--James and Springer orders}

Recall the Dipper--James order $\leq_{DJ}$ from Definition \ref{def:DJ}, and the Springer order (for the good characteristic case) $\leq^\NC_\triv$ defined in \S \ref{subsec:good}.
We will also need the dominance order on compositions: for $x\in \NM^t$, we let $S_i(x) = \sum_{i \leq k \leq t} x_k$; for $x,y\in \NM^t$, we write $x\leq y$ if and only if $\ S_i(x) \leq S_i(y)$ for all $i$.

\begin{proposition}
\label{prop_Cn}
We have $\leq_{DJ} {}\subset{} \leq^\NC_\triv$.
\end{proposition}

\begin{proof}

Let $\lamb,\mub\in\Irr\KM W = \Bipart_W$, and let $\lambda,\mu\in\PC(G)$ be the partitions of the corresponding nilpotent  orbits (we may ignore decorations for very even orbits). We want to show
\[
\bfl <_{DJ} \bfm \Longrightarrow \lambda < \mu.
\]

We write $\lambda^{(i)}$ and $\mu^{(i)}$ for the components of $\bfl$ and $\bfm$. In the $G = \SO_{2n}$ case, by definition of the Dipper--James order on unordered bipartitions, we can assume (up to flipping $\bfl$) that $(\lambda^{(1)},\lambda^{(2)}) \leq_{DJ} (\mu^{(1)},\mu^{(2)})$ (in $\Bipart_n$ rather than $\widetilde\Bipart_n$).

As in the proof of Proposition \ref{prop:DJ delta}, we may assume that $\lambda^{(2)} = \mu^{(2)}$, and that 
$\lambda^{(1)}$ is obtained from $\mu^{(1)}$ by increasing one part by $1$ and decreasing another part by $1$. Again, if we
reorder the entries of the symbols $\Phi^G_\co(\bfl)$ and $\Phi^G_\co(\bfm)$ into the respective sequences $x = ( x_{1}\leq x_{2} \leq \dots \leq x_{t} )$ and $y = ( y_{1} \leq y_{2} \leq \dots \leq  y_{t} )$, i.e. $x = w(\dist \Phi^G_\co(\bfl) )$ and $y = w( \dist \Phi^G_\co(\bfm) )$, then there exist indices $a < b$ such that $y_{b} = x_{b} + 1$ and $y_{a} = x_{a} - 1$, while $y_{i} = x_{i}$ for all $i \neq a, b$. This implies that for $1\leq i \leq t$, we have
\[
S_{i}(y) - S_{i}(x) = \ind_{]a,b]}(i).
\]

Let $\tilde x = ( \tilde x_{1}\leq \tilde x_{2} \leq \dots \leq \tilde x_t )$ and $\tilde y = ( \tilde y_{1} \leq \tilde y_{2} \leq \dots \leq  \tilde y_t )$ be the sequences obtained from $x$ and $y$ as in \S \ref{conv}, so that $\lambda = 2( \tilde x - w( \Lambda_{G} ) )$ and $\mu = 2( \tilde y - w( \Lambda_{G} ) )$. Hence to prove that $\lambda < \mu$, it is enough to prove that $\tilde x < \tilde y$. Moreover, proving $\tilde x \leq \tilde y$ is enough: we know that that $\lambda$ and $\mu$ cannot be equal, since by \S \ref{sec:mod}, $\dim\OC^G_\lambda < \dim\OC^G_\mu$.

First consider the case $G = \Sp_{2n}$. Recall that $\tilde x$ and $\tilde y$ are obtained from $x$ and $y$ by replacing pairs $( s, s )$ of equal consecutive elements by $( s - 1/2, s + 1/2 )$. Then for all $1\leq i \leq t$,
\[
S_{i}(\tilde x) = S_i(x) + \frac 1 2 \delta_{x_{i-1}, x_{i}}
\quad \text{ and }\quad
S_{i}(\tilde y) = S_i(y) + \frac 1 2 \delta_{y_{i-1}, y_{i}}.	
\]
We get
\[
S_{i}(\tilde y) - S_{i}(\tilde x)
= S_{i}(y) - S_{i}(x) + \frac 1 2 ( \delta_{y_{i-1}, y_{i}} - \delta_{x_{i-1}, x_{i}} )
= \ind_{]a,b]}(i) + \frac 1 2 ( \delta_{y_{i-1}, y_{i}} - \delta_{x_{i-1}, x_{i}} ).
\]

Suppose $S_{i}(\tilde y) - S_{i}(\tilde x) < 0$ for some $i$. Then we must have $i \notin ]a,b]$, $y_{i-1}\neq y_{i}$, and $x_{i-1} = x_{i}$, which can only happen if $i = a$ and $y_{a-1} = y_{a} + 1$, or $i = b + 1$ and $y_{b} - 1 = y_{b+1}$. In both cases, we derive a contradiction for the entries of $y$ weakly increase from left to right. Since $S_{i}(\tilde x ) \leq S_{i}(\tilde y)$ for any $1\leq i \leq t$, we conclude that $\tilde x \leq \tilde y$ and we are done.

\medskip

Now assume $G = {\mathbf{SO}}_{2n+1}$. Then $t = 2m + 1$. According to
\S \ref{conv}, we can form disjoint pairs $(y_{i-1},y_{i})$ of
consecutive letters in $y$ such that $y_{i-1} = y_{i}.\ $For such a
pair we have
\[
( \tilde y_{i-1}, \tilde y_{i} ) =
\begin{cases}
( y_{i - 1}, y_{i}) & 
	\text{ if } y_{i - 1} = y_{i} \text{ and } i \text{ is even,}\\
(y_{i-1}-1,y_{i}+1) &
	\text{ if } y_{i - 1} = y_{i} \text{ and } i \text{ is odd.}
\end{cases}
\]

Let $K$ be the set of indices $i$ such that $y_{i}$ belongs to one of these pairs. Now for any $i\notin K$
\[
\tilde y_{i}=\left\{
\begin{array}[c]{l}
y_{i} - 1/2 \text{ if }i\text{ is even,} \\
y_{i} + 1/2 \text{ if }i\text{ is odd.}
\end{array}
\right.
\]
Consider $i\in \{ 1, \ldots, 2m + 1 \}$. When $i$ is even, $( y_{i - 1}, y_{i} )$ is not a distinguished pair. Moreover, the set $U_{i}=\{j\in\{i,\ldots,2m\}\mid j\notin K\}$ contains the same number of even and odd indices. Thus $S_{i}(\tilde{y})=S_{i}(y).$
When $i$ is odd, by computing the number of odd and even indices in $U_{i}$ in both cases $(y_{i-1},y_{i})$ distinguished or not, we obtain
\[
S_{i}(\tilde{y})=S_{i}(y)+\frac{1}{2}\delta_{y_{i-1},y_{i}}+\frac{1}{2}.
\]
We have the same relation between $S_{i}(\tilde{x})$ and
$S_{i}(x)$. Moreover $S_{i}(x)=S_{i}(y
)-\boldsymbol{1}_{]a,b]}(i)$. This gives
\[
S_{i}(\tilde{x})=\left\{
\begin{array}[c]{l}
S_{i}(\tilde{y})-\boldsymbol{1}_{]a,b]}(i)\text{ if }i\text{ is even,}\\
S_{i}(\tilde{y})-\boldsymbol{1}_{]a,b]}(i)+\frac{1}{2}(\delta_{x
_{i-1},x_{i}}-\delta_{y_{i-1},y_{i}})\text{ if
}i\text{ is odd.}
\end{array}
\right.
\]
When $i$ is even, we have clearly $S_{i}(\tilde{x})\leq
S_{i}(\tilde{y}).$ This is also true for $i$ odd by using the same
arguments as for $\Sp_{2n}$.

\medskip

For $G = \SO_{2n}$, we proceed as in the case $G = {\mathbf{SO}}_{2n+1}$, switching the cases $i$ even and $i$ odd due to the fact that $t=2m$ is even. 
\end{proof}

\end{document}